\documentclass[a4paper]{article}
\usepackage{amsmath,amssymb,amsfonts,amsthm} 
\usepackage{graphicx}
\usepackage{appendix}

\usepackage{subcaption}
\usepackage{mathrsfs}
\usepackage{booktabs}  
\usepackage{wrapfig}
\usepackage{url} \usepackage[usenames]{color}
\def\b1{\mbox{\boldmath $1$}}

\newtheorem{example}{Example}[section]

\theoremstyle{plain}
\newtheorem{thm}{\bf Theorem}[section]

\newtheorem{lem}[thm]{\bf Lemma}

\theoremstyle{remark}

\numberwithin{figure}{section}
\numberwithin{subfigure}{section}
\numberwithin{table}{section}
\numberwithin{equation}{section}

\makeatletter \@addtoreset{equation}{section} \makeatother \makeatletter
\@addtoreset{thm}{section} \makeatother
\allowdisplaybreaks

\newcommand{\customfootnotetext}[2]{{
		\renewcommand{\thefootnote}{#1}
		\footnotetext[0]{#2}}}

\usepackage[numbers,sort&compress]{natbib} 
\usepackage[bookmarksnumbered, bookmarksopen,
colorlinks,citecolor=blue,linkcolor=blue]{hyperref}

\title{\Large \bf Optimal investment problem in a renewal risk model
	with generalized Erlang distributed interarrival times
	\thanks{This work was supported by the National Natural Science Foundation of China (No. 12201104).}}
\author{
	{Linlin Tian}$^a$\footnote{E-mail: linlin.tian@dhu.edu.cn}~~~{Yixuan Tian}$^a$\footnote{E-mail: 2232539@mail.dhu.edu.cn}~~~{Bohan Li}$^b$\footnote{E-mail: bhli@suda.edu.cn}~~~{Guoqing Li}$^a$\footnote{E-mail: 2212267@mail.dhu.edu.cn}~~~
	\\
	\\
	\small  a. School of Mathematics and Statistics, Donghua University,
	Shanghai, P.R. China 201620 \\
	\small  b. Center for Financial Engineering, Soochow University, Suzhou, Jiangsu, P.R. China   215006
}

\begin{document}
\date{}
\maketitle

\customfootnotetext{$ $}{Corresponding author: Bohan Li}

\begin{center}
\begin{minipage}{130mm}
{\bf Abstract} \;\;
This paper explores the optimal investment problem of a renewal risk model with generalized Erlang distributed interarrival times. The phases of the Erlang interarrival time is assumed to be observable.   The price of the risky asset is driven by the constant elasticity of variance   model (CEV) and the insurer aims to maximize the exponential utility of the terminal wealth  by asset allocation.   By solving the corresponding Hamilton-Jacobi-Bellman (HJB) equation, we establish the concavity of the value function and derive an explicit expression for the optimal investment policy when the interest rate is zero. When the interest rate is nonzero, we obtain an explicit form of the optimal investment strategy, along with a semi-explicit expression of the value function, whose concavity is also rigorously proven.

\vspace{3mm} {\bf Keywords:} Exponential utility;  Renewal process; Stochastic optimal control; Hamilton-Jacobi-Bellman equation
\end{minipage}
\end{center}
\section{Introduction}\label{2323}
The optimal investment problem  of a general insurer has been studied under various settings since the work of
\cite{Hipp2000}.
With the exponential utility,
\cite{Yang2005}
considers the optimal investment strategy for an insurer with jump-diffusion surplus process
when the risky asset follow a Geometric Brownian motion in which the explicit expression of the optimal investment strategy is shown.
 \cite{Wang2007} extends the results of
\cite{Yang2005} to the case of multiple risky assets.
 \cite{Chiu2013} studies the optimal investment for an insurer with cointegrated assets under CRRA utility .
\cite{Gu2012} studies the excess-of-loss reinsurance and investment problems under a constant elasticity of variance (CEV) model. \cite{Zheng2016} studies the same problem as \cite{Gu2012} with the consideration of model misspecification.

The above works are investigated in the Markovian framework where, for the compound Poisson model, the interclaim times are  exponential distributed. However, it is not the case of industrial practice, as the exponential distributed interclaim times have no memory about the time elapsed since the last claim. To overcome such drawback, the renewal process  was introduced to take place of Poission  process in characterizing  the surplus of the insurance company. For example, \cite{Albre2005,Albrecher2006}
 introduce the Erlang($n$) as the distribution of the interarraival time. They calculate the  moment-generating function of the discounted dividends under a horizontal barrier strategies  and show that, in general, a horizontal barrier strategy  is not the optimal dividend strategy.
 \cite{Mish2012} shows that the optimal dividend strategy is of phase-wise barrier strategy for the Erlang($n$) interclaim times. As consequential studies, \cite{Bai2017,Bai2021} investigate the optimal investment and dividend problem under Sparre Andersen model by applying the theory of viscosity solution.
 We refer the readers to \cite{Bergel2015,Li2006} for more optimal control problems about renewal surplus process. To our knowledge, there is generally no explicit solution to the optimization problem involving the renewal process. Thus, the paper explores  the optimal investment problem for a insurance company under a Erlang($n$) distributed interclaim renewal process and study its explicit or semi-explicit solutions.

We introduce the following the renewal claim process.
Let  $\{J_t\}$ be a homogenous Markov chain on the phase state space $\{1,2,\cdots, n\}$ with an intensity matrix of the form
\begin{align*}
\left(
  \begin{array}{ccccc}
    -\lambda_1 & \lambda_1 & 0 & \cdots & 0 \\
    0 & -\lambda_2 & \lambda_2 & \cdots & 0 \\
    \vdots & \vdots & \vdots & \ddots & \vdots \\
    \lambda_n & 0 & 0 & \cdots & -\lambda_n \\
  \end{array}
\right),
\end{align*}
i.e., the phase moves through $1\rightarrow 2 \rightarrow\cdots\rightarrow n\rightarrow 1\rightarrow \cdots$ and stays in phase $i\in\{1,2,\cdots, n\}$ for the intensity with parameter $\lambda_i$ which is also called an exponential clock. When the exponential clock rings,
$J_t$ jumps to the next phase. Specifically,
 if the current phase is $n$, after waiting for  an exponential time with parameter $\lambda_n$, $J_t$ will jump to phase $1$. We assume the claims only occurs when the phase jumps from $n$ to 1. If $n=1$, the renewal process degenerates to the compound Poisson process.

We assume that the phases of $J_t$  can be observed. For example, the  lithium battery-powered electric bicycles are common vehicle in China. Since unstable voltage can cause lithium batteries to explode or catch fire during charging, electric vehicles with lithium batteries are prohibited from being charged indoors.  Each  incident caused by the indoor charging electric bicycle generates extensive coverage across various media platforms, such as TV news, TikTok and X, thereby raising public awareness of such incidents. The government also issues a series of regulations and guidelines to avoid future possible accidents.
At this point, insurance companies can infer from public opinion and human memory of the last disaster that they are in phase 1, where the probability of new claims arising from such accidents is virtually zero. As time passes, public concern diminishes, and platforms stop promoting awareness, though the government may continue to run advertisements on the safe charging requirement for lithium batteries. This marks the transition to phase 2.
As time continues to pass, people's awareness fades, and advertisements expire, and the state turns to phase 3, where the insurance company will face more claim risk. As a result of the above discussion, it is reasonable to assume that the insurance company can observe the phases of $J_t$.

As the Markov chain $J_t$ has   $n$ phases,   the Hamilton-Jacobi-Bellman (HJB) are a system
of $n$-dimensional coupled  equations.
We consider the optimization problem in two cases: interest rate is zero, and interest rate is non-zero. The solutions to the HJB equation take different forms in  two cases. The key challenge lies in proving the concavity of each solution. For the first case, we establish concavity using the Laplace transform of the Markov chain with killing. Furthermore, we derive explicit expressions for both the value function and the optimal investment strategy.
For the second case, expressing the HJB solution it is more challenging to derive an explicit solution to the HJB equation.
We rigorously prove the existence of a solution and show a semi-explicit expression for the solution of the HJB equation.
Furthermore, the concavity of the value function is established by  applying the decoupling argument and  the Banach fixed-point theorem.

Let us compare the Erlang($n$) renewal model with the regime-switching model.
For discussions on regime-switching, refer to \cite{Elliott2011},  \cite{Jiang2012}, \cite{Liu2014}, \cite{zou2014}.
The regime-switching model assume that the model parameters such as the interest rate, claim intensity, return rate, and volatility rate changes according to the current regime, while the regime switching typically do not lead to a direct reduction in wealth. In contrast, the Erlang($n$) model allows both changes in parameters and wealth upon the phase entering a new phase.
Specifically, when the phase jumps from phase $n$ to phase 1, the intensity of interarrival time will turns to that of phase 1 and the wealth will reduce by the size of a claim.
Our present article focus on the optimal investment problem in the context of Erlang($n$) interclaim times.

The main contributions of this work are twofold:
(1) We extend the classical result of optimal investment problems with compound Poisson claims to that of a renewal process with Erlang-distributed interarrival time.  The explicit and semi-explicit solutions to the problems are provided. (2)
The concavity of solutions to the corresponding HJB equation is established. To this end, we employ two different approach for two cases. If the interest rate is zero, the concavity can be proven by applying the Laplace transform to a Markov chain with killing. If the interest rate is a positive constant, we resort to the decoupling argument and Banach fixed point theorem.

The structure of the paper is as follows. Section \ref{model1} introduces the surplus process of the insurance company. The goal of the insurance company is find an optimal investment strategy to maximize the expected exponential utility of terminal wealth. The problem is  divided into two cases. Section \ref{ere2121} studies the case that the interest rate is zero, presenting the explicit forms of both the optimal value function and the optimal strategy. In Section \ref{intere2}, we explore the case that the interest rate is non-zero. We provide the explicit form of the optimal investment policy, and by applying  the decoupling argument, along with the Banach fixed point theorem, we demonstrate the concavity of the solution to the HJB equation.
Section \ref{5} analyzes the sensitivity of optimal strategies and value functions on the model parameters Finally, Section \ref{conclusion} concludes the paper. Some proofs are presented in the Appendix.

\section{Modelling}\label{model1}
We work on a complete probability space $(\Omega,\mathscr{F}, \mathbb{P})$ which satisfies the usual condition. Let $T>0$ be the planning horizon and $\mathscr{F}_t$ stands for the information available before time $t\in[0,T].$
The surplus process $\{C_t, t\ge 0\}$  is defined by
\begin{align}
dC_t=cdt-d\sum_{i=1}^{N_t}Y_i,
\end{align}
where $c>0$ is the premium rate, $N_t$ is a renewal counting process representing the number of claims before time $t$, $\{Y_i\}_{i=1}^{\infty}$ are independent and identically distributed positive random variables, and $Y_i$ represents the size of the $i$-th claim.
 The interclaim times are assumed to be independently generalized Erlang($n$)  distributed.
  We  assume  the state of the Markov chain behind the Erlang($n$) distribution can be observed.
 We also assume that the claim size $\{Y_i\}_{i=1}^{\infty}$  are independent of the markov chain $\{J_t\}$. In addition, we following \cite{Ji2014} to assume the premium $c>\frac{\mathbb{E}(Y_i)}{\sum_{i=1}^n {\lambda_i}^{-1}}$
  which implies that the insurance business tends to be profitable in the long run and ensures that the insurance company's ruin probability is strictly less than one.

The insurer is allowed to invest in a financial market consisting two assets, one risk-free asset and one risky asset. The price process $P_t$ of the risk-free asset follows
\[dP_t=rP_tdt,\]
where $r>0$ is the risk-free interest rate.  We assume that the price process of the risky asset is driven by the CEV model
\[dS_t=S_t(\mu dt+\sigma S_t^\beta dW_t),\]
where $\mu>r$ is the expected instantaneous return rate of the risky asset, $\sigma >0$ is a constant, $\beta\ge 0$ represents the elasticity parameter, and $W_t$ is a standard Brownian motion which is independent of the compound renewal process. When $\beta=0,$ the CEV model degenerates to a geometric Brownian motion.

The investment strategy is denoted by $\{a_t\}_{0\le t\le T}$, where $a_t\in\mathbb{R}$ denotes the total amount of money invested in the risky asset. Under the strategy $a$,   the surplus process of the insurance company follows
\begin{align}\label{sa1wqw}
dX_t^a=(rX_t+(\mu-r)a_t+c)dt+\sigma S_t^\beta a_tdW_t-d\sum_{i=1}^{N_t}Y_i.
\end{align}
We call a strategy $\{a_t\}_{0\le t\le T}$ admissible if for any $t\in[0,T]$, $a_t$  is $\mathscr{F}_t$
progressively measurable, $\mathbb{E}[\int_0^{+\infty}a_t^2S_t^{2\beta}dt]<+\infty$ and the equation \eqref{sa1wqw} admits a unique strong solution. Denote $\mathscr{U}_{ad}$ the set of all admissible strategies.

We consider a optimal investment problem maximizing the expected exponential utility of the terminal wealth. The exponential  utility function is defined as
\[U(x)=-\frac{1}{m}e^{-mx },\]
where $m>0$ is a constant which is known as the absolute risk aversion parameter. Such  utility functions play a crucial role in the field of mathematical finance and actuarial science. On one hand, the exponential utility function offers analytical tractability; on the other hand, \cite{Bu2021} illustrates that by adjusting the parameter
$m$, the exponential utilities can interpolate between a risk-sensitive criterion and a robust criterion.

 For mathematical simplicity, we assume that, for any $i$, $\mathbb{E}(e^{mY_ie^{rT}})<+\infty$. This condition eliminates  the extreme scenario where the expected claim size is too large such that the insurance company could face immediate ruin. If the claim size $Y_i$ follows the exponential distribution with the parameter $\lambda$, then the condition  $\mathbb{E}(e^{mY_ie^{rT}})<+\infty$ is equivalent to $m e^{rT}<\lambda$. In the following context, we will omit the subscript $i$ if there is no ambiguity.

For any investment strategy $\{a_t\}_{0\le t\le T}\in \mathscr{U}_{ad}$ and any initial state $(s,x,i)$,  define the utility of the strategy $a$ as
\begin{align}
J(t,x,s,i;a)=\mathbb{E}\left[U(X_T^a)|X_t=x,S_t=s, J_t=i\right].
\end{align}
The value function is defined as
\begin{align}\label{sszx1223}
V(t,x,s,i)=\sup_{\{a_t\}\in\mathscr{U}_{ad}}J(t,x,s,i;a),
\end{align}
The aim of our paper is to find the optimal policy $a_t^*\in\mathscr{U}_{ad}$ so that $J(t,x,s,i;a^*)=V(t,x,s,i)$.
By dynamic programming principle, we derive the HJB equation of the optimization problem \eqref{sszx1223} which is an $n$-dimensional coupled equation:
\begin{align}\label{hjb01}
\begin{cases}
&v_t(t,x,s,i)+\sup_{a\in\mathbb{R}}\big\{v_x(t,x,s,i)(c+a(\mu-r)+rx)+\frac{1}{2}\sigma^2a^2s^{2\beta}v_{xx}(t,x,s,i)\\
&\quad +\sigma^2as^{2\beta+1}v_{xs}(t,x,s,i)\big\}
+\mu sv_s(t,x,s,i)+\frac{1}{2}\sigma^2s^{2\beta+2}v_{ss}(t,x,s,i)\\
&\quad+\lambda_i(v(t,x,s,i+1)-v(t,x,s,i))=0,\quad i=1,2,\cdots, n-1;\\
&v_t(t,x,s,i)+\sup_{a\in\mathbb{R}}\big\{v_x(t,x,s,i)(c+a(\mu-r)+rx)+\frac{1}{2}\sigma^2a^2s^{2\beta}v_{xx}(t,x,s,i)\\
&\quad+\sigma^2as^{2\beta+1}v_{xs}(t,x,s,i)\big\}+\mu sv_s(t,x,s,n)+\frac{1}{2}\sigma^2s^{2\beta+2}v_{ss}(t,x,s,i)\\
&\quad+\lambda_n(\mathbb{E}[v(t,x-Y,s,1)]-v(t,x,s,i))=0, \quad i=n,
\end{cases}
\end{align}
with boundary condition
\begin{align}\label{hjb02}
v(T,x,s,i)=-\frac{1}{m} e^{-mx},\quad i=1,2,\ldots, n.
\end{align}

We tackle with the problem for  the case of $r=0$ and $r\neq 0$ by different approach. There exists an explicit solution for the case of $r=0$, while no explicit form of solutions for the case of  $r\neq 0$. For the latter case, the solutions to the HJB equations can be expressed as a semi-explicit form.
We first consider the case $r=0$ in the next section.
\section{When the interest rate is 0}\label{ere2121}
If the interest rate is 0,  equation \eqref{hjb01} reduces to the following equations:
\begin{align}\label{hjb0180786}
\begin{cases}
&v_t(t,x,s,i)+\sup_{a\in\mathbb{R}}\big\{v_x(t,x,s,i)(c+a\mu)+\frac{1}{2}\sigma^2a^2s^{2\beta}v_{xx}(t,x,s,i)\\
&\quad +\sigma^2as^{2\beta+1}v_{xs}(t,x,s,i)\big\}
+\mu sv_s(t,x,s,i)+\frac{1}{2}\sigma^2s^{2\beta+2}v_{ss}(t,x,s,i)\\
&\quad+\lambda_i(v(t,x,s,i+1)-v(t,x,s,i))=0,\quad i=1,2,\cdots, n-1;\\
&v_t(t,x,s,i)+\sup_{a\in\mathbb{R}}\big\{v_x(t,x,s,i)(c+a\mu)+\frac{1}{2}\sigma^2a^2s^{2\beta}v_{xx}(t,x,s,i)\\
&\quad+\sigma^2as^{2\beta+1}v_{xs}(t,x,s,i)\big\}+\mu sv_s(t,x,s,i)+\frac{1}{2}\sigma^2s^{2\beta+2}v_{ss}(t,x,s,i)\\
&\quad+\lambda_n(\mathbb{E}[v(t,x-Y,s,1)]-v(t,x,s,i))=0, \quad i=n.
\end{cases}
\end{align}
Note that if there exists a  concave  solution for \eqref{hjb0180786}, for each $i,$  the maximizer of \eqref{hjb0180786} can be obtained as $a^*(t,x,s,i) :=-\frac{\mu v_x(t,x,s,i)+\sigma^2s^{2\beta+1}v_{xs}(t,x,s,i)}{\sigma^2s^{2\beta}v_{xx}(t,x,s,i)}$. In what follows, we look for a continuously differentiable concave solution for \eqref{hjb0180786}.
We take the following {\it ansatz} of the solution:
\begin{align}\label{1123019456de}
\begin{split}
v(t,x,s,i)=-\frac{1}{m}\exp\left\{-m x+\frac{\mu^2}{2\sigma^2} (t-T)s^{-2\beta}\right\}\psi_i(t),i=1,2,\ldots, n,
\end{split}
\end{align}
where $\psi_i(t), i=1,2,\cdots, n$ are some  functions to be determined later.
After direct calculations,
\begin{equation}\label{dsid2313901}
\left\{\begin{array}{l}
v_t=v(t,x,s,i)\frac{\mu^2}{2\sigma^2}s^{-2\beta}-\frac{1}{m}\exp\left\{-m x+\frac{\mu^2}{2\sigma^2} (t-T)s^{-2\beta}\right\}\psi_i'(t),\\
v_x=-m v(t,x,s,i),v_{xx}=m^2v(t,x,s,i),\\
v_s=\beta \frac{\mu^2(T-t)}{\sigma^2}s^{-2\beta -1}v(t,x,s,i),\\
v_{ss}=[\frac{\beta^2  \mu^4 (T-t)^2 }{\sigma^4}s^{-4\beta -2}-\frac{\beta \mu^2(T-t)}{\sigma^2}(2\beta+1)s^{-2\beta -2}]v(t,x,s,i),\\
v_{xs}=-m\beta \frac{\mu^2(T-t)}{\sigma^2}s^{-2\beta -1}v(t,x,s,i).
\end{array}\right.
\end{equation}
Substituting
\begin{align}\label{23u2i2323}
a^*(t,x,s,i)=-\frac{\mu v_x(t,x,s,i)+\sigma^2s^{2\beta+1}v_{xs}(t,x,s,i)}{\sigma^2s^{2\beta}v_{xx}(t,x,s,i)}
 \end{align}
 and \eqref{dsid2313901} into \eqref{hjb0180786} and eliminating same terms
show that  $\{\psi_i\}_{i=1}^n$ satisfies
\begin{align}\label{463583021}
\begin{cases}
\psi_1'(t)-(cm+\frac{(2\beta +1)\mu^2 \beta (T-t)}{2}+\lambda_1)\psi_1(t)+\lambda_1\psi_2(t)=0,\\
\psi_2'(t)-(cm+\frac{(2\beta +1)\mu^2 \beta (T-t)}{2}+\lambda_2)\psi_2(t)+\lambda_2\psi_3(t)=0,\\
\cdots\\
\psi_{n-1}'(t)-(cm+\frac{(2\beta +1)\mu^2 \beta (T-t)}{2}+\lambda_{n-1})\psi_{n-1}(t)+\lambda_{n-1}\psi_{n}(t)=0,\\
\psi_{n}'(t)-(cm+\frac{(2\beta +1)\mu^2 \beta (T-t)}{2}+\lambda_{n})\psi_n(t)+\lambda_{n}\psi_{1}(t)\mathbb{E}(e^{mY})=0.
\end{cases}
\end{align}
Denote   $L(t):=(cm+\frac{2(\beta +1)\mu^2 \beta T)}{2}) t -\frac{2(\beta +1)\mu^2 \beta}{4}t^2$.  Multiplying $e^{-L(t)}$ on both sides of \eqref{463583021} gives
\begin{align*}
\begin{cases}
\varphi_1'(t)-\lambda_1\varphi_1(t)+\lambda_1 \varphi_2(t)=0,\\
\varphi_2'(t)-\lambda_2\varphi_2(t)+\lambda_2 \varphi_3(t)=0,\\
\cdots\\
\varphi_{n-1}'(t)-\lambda_{n-1}\varphi_{n-1}(t)+\lambda_{n-1}\varphi_n(t)=0,\\
\varphi_n'(t)-\lambda_n \varphi_n(t)+\lambda_n \mathbb{E}(e^{mY})\varphi_1(t)=0,
\end{cases}
\end{align*}
or equivalently,
\begin{align}\label{78e0e0367433}
\left(
  \begin{array}{c}
    \varphi_1'(t) \\
    \varphi_2'(t) \\
    \vdots\\
    \varphi_{n-1}'(t) \\
    \varphi_n'(t) \\
  \end{array}
\right)=\hat{Q}\left(
                 \begin{array}{c}
                   \varphi_1(t) \\
                   \varphi_2(t) \\
                   \vdots \\
                 \varphi_{n-1}(t) \\
                 \varphi_n(t) \\
                 \end{array}
               \right),
\end{align}
where
\begin{align*}
\hat{Q}=\left(
       \begin{array}{cccccc}
         \lambda_1 & -\lambda_1 & 0 & \cdots & 0 & 0 \\
          0&  \lambda_2& -\lambda_2 &\cdots  &0 & 0 \\
          0&  0&\lambda_3  & \cdots & 0 & 0 \\
          \vdots& \vdots & \vdots & \vdots & \vdots & \vdots \\
          0& 0 & 0 & \cdots & \lambda_{n-1} &-\lambda_{n-1}  \\
        -\lambda_n \mathbb{E}(e^{mY}) &0  &0  &\cdots  & 0 & \lambda_n \\
       \end{array}
     \right).
\end{align*}
The boundary condition of $\varphi_i(t)$ is
\begin{align} \label{787823eeeqw2}
\varphi_i(T)=e^{-L(T)}>0.
\end{align}
Combining the boundary condition  \eqref{787823eeeqw2}, the explicit solution of $\varphi_i(t)$ can be derived.
 \begin{align}\label{4.tterw1516576}
\left(
  \begin{array}{c}
    \varphi_1(t) \\
    \varphi_2(t) \\
    \vdots\\
    \varphi_{n-1}(t) \\
    \varphi_n(t) \\
  \end{array}
\right)=e^{\hat{Q}(t-T)}\left(
                 \begin{array}{c}
                   \varphi_1(T) \\
                   \varphi_2(T) \\
                   \vdots \\
                 \varphi_{n-1}(T) \\
                 \varphi_n(T) \\
                 \end{array}
               \right)=e^{\hat{Q}(t-T)}\left(
                 \begin{array}{c}
                  e^{-L(T)}\\
                   e^{-L(T)} \\
                   \vdots \\
                e^{-L(T)} \\
                 e^{-L(T)} \\
                 \end{array}
               \right),
 \end{align}
Now we show that the solution $\varphi_i(t), i=1,2\cdots, n,$ are non-negative. It suffices to show that every element of the matrix $e^{\hat{Q}(t-T)}$ is non-negative.
\begin{lem}\label{e2daw21212}
Every element of the matrix $e^{\hat{Q}(t-T)}$ is non-negative.
\end{lem}
The matrix $e^{\hat{Q}(t-T)}$  can be seen as the transition matrix of a Markov chain with killing.
The detailed proof in the Appendix \ref{appendixA}.

\begin{lem}
The function $v(t,x,s,i)$ is concave in $x$.
\end{lem}
\begin{proof}
By the above argument, we can notice that the concavity of  $v(t,x,s,i)$ is equivalent to the non-negativity of $\{\varphi_i(t)\}_{i=1}^n$,  which can be proven by applying Lemma \ref{e2daw21212}.
\end{proof}

At this point, we obtain a continuously differentiable concave solution to the HJB equation:
\begin{align}\label{1123019456de67w2}
\begin{split}
v(t,x,s,i)=-\frac{1}{m}\exp\left\{-m x+\frac{\mu^2}{2\sigma^2} (t-T)s^{-2\beta}\right\}\psi_i(t), i=1,2,\ldots, n,
\end{split}
\end{align}
where $\psi_i(t)=\varphi_i(t)e^{L(t)}$ and $\varphi_i(t)$ is given by \eqref{4.tterw1516576}.

Now we provide the verification theorem:
\begin{thm}\label{475sss}(verification theorem) Denote $
\Gamma :=\frac{4 \mu^{2}+12 \mu^{3} \beta T+8 \mu^{4} \beta^{2} T^{2}}{\sigma^{2}},$ if either one of the following conditions is satisfied:
\begin{enumerate}\label{r=0condition}
    \item$\Gamma\leq \frac{\mu^2}{2\sigma^2}$,
    \item $\Gamma>\frac{\mu^2}{2\sigma^2} $, and $ T<\gamma_2^{-1}\operatorname{arccot}(\frac{-\gamma_1}{\gamma_2})$, where $\gamma_1=\mu \beta, \gamma_2=\frac{\sqrt{-4\mu^2\beta^2+8\beta^2 \sigma^2\Gamma}}{2},$

\end{enumerate}
 then the value function  $V(t,x,s,i)=v(t,x,s,i),$ where
 $v$ is shown in \eqref{1123019456de67w2}. The optimal investment policy is \begin{equation*}\label{r=0a*}
  a^{*}_t=\frac{\mu +\mu^2\beta(T-t)}{\sigma^{2} s^{2 \beta} m}.
\end{equation*}
\end{thm}
The detailed proof is in the Appendix \ref{appendixB}.

\section{When the interest rate is  not 0}\label{intere2}
When the interest rate is not $0$, the HJB equation is
\begin{align}\label{hjb01001}
\begin{cases}
&v_t(t,x,s,i)+\sup_{a\in\mathbb{R}}\big\{v_x(t,x,s,i)(c+a(\mu-r)+rx)+\frac{1}{2}\sigma^2a^2s^{2\beta}v_{xx}(t,x,s,i)\\
&\quad +\sigma^2as^{2\beta+1}v_{xs}(t,x,s,i)\big\}
+\mu sv_s(t,x,s,i)+\frac{1}{2}\sigma^2s^{2\beta+2}v_{ss}(t,x,s,i)\\
&\quad+\lambda_i(v(t,x,s,i+1)-v(t,x,s,i))=0,\quad i=1,2,\cdots, n-1;\\
&v_t(t,x,s,i)+\sup_{a\in\mathbb{R}}\big\{v_x(t,x,s,i)(c+a(\mu-r)+rx)+\frac{1}{2}\sigma^2a^2s^{2\beta}v_{xx}(t,x,s,i)\\
&\quad+\sigma^2as^{2\beta+1}v_{xs}(t,x,s,i)\big\}+\mu sv_s(t,x,s,n)+\frac{1}{2}\sigma^2s^{2\beta+2}v_{ss}(t,x,s,i)\\
&\quad+\lambda_n(\mathbb{E}[v(t,x-Y,s,1)]-v(t,x,s,i))=0, \quad i=n,
\end{cases}
\end{align}
with boundary condition
\begin{align}\label{hjb02001}
v(T,x,s,i)=-\frac{1}{m} e^{-mx},\quad i=1,2,\ldots, n.
\end{align}
If there exists a  concave  solution to \eqref{hjb01001},  the maximizer of \eqref{hjb01001} is given by
\begin{align}\label{657657}
a^*(t,x,s,i)=-\frac{(\mu-r)
v_x(t,x,s,i)+\sigma^2s^{2\beta+1}v_{xs}(t,x,s,i)}{\sigma^2s^{2\beta}v_{xx}(t,x,s,i)}, i=1,2,\cdots,n.
\end{align}
Similar to Section \ref{ere2121}, we look for a continuously differentiable concave solution to the HJB equation
\eqref{hjb01001}-\eqref{hjb02001}. 
We take an {\it anstaz}:
\begin{align}\label{112301}
\begin{split}
v(t,x,s,i)=-\frac{1}{m}\exp\left\{-m x e^{r(T-t)}-\frac{(\mu - r)^2}{4\sigma^2\beta r}[1-e^{2\beta r(t-T)}]s^{-2\beta}\right\}\psi_i(t),i=1,2,\ldots, n.
\end{split}
\end{align}
where $\psi_i(t), i=1,2,\cdots, n$, are some deterministic function which will be determined later.
By taking derivatives, we have
\begin{align}\label{sderive1}
\begin{cases}
v_t(t,x,s,i)=v(t,x,s,1)\left(mrx e^{r(T-t)}+\frac{(\mu - r)^2}{2\sigma^2}e^{2\beta r(t-T)}s^{-2\beta}\right)
\\ \qquad\qquad\quad-\frac{1}{m}\exp\left\{-m x e^{r(T-t)}-\frac{(\mu - r)^2}{4\sigma^2\beta r}[1-e^{2\beta r(t-T)}]s^{-2\beta}\right\}\psi_i'(t),\\
v_{x}(t,x,s,i)=v(t,x,s,i)\left\{-m e^{r(T-t)}\right\},
v_{xx}(t,x,s,i)=v(t,x,s,i)\left\{-m e^{r(T-t)}\right\}^2,\\
v_{s}(t,x,s,i)=v(t,x,s,i)\left\{ \frac{(\mu - r)^2}{2\sigma^2r}[1-e^{2\beta r(t-T)}]s^{-2\beta-1}\right\},\\
  v_{ss}(t,x,s,i)=v(t,x,s,i)\left\{ \frac{(\mu - r)^4}{4\sigma^4r^2}[1-e^{2\beta r(t-T)}]^2s^{-4\beta-2}\right\}
\\ \qquad\qquad\quad\quad+ v(t,x,s,i)\left\{-(2\beta+1) \frac{(\mu - r)^2}{2\sigma^2 r}[1-e^{2\beta r(t-T)}]s^{-2\beta-2}\right\},\\
  v_{xs}(t,x,s,i)=v(t,x,s,i)\left\{-m e^{r(T-t)} \frac{(\mu - r)^2}{2\sigma^2 r}[1-e^{2\beta r(t-T)}]s^{-2\beta-1}\right\}.
\end{cases}
\end{align}
Substituting \eqref{657657} and \eqref{sderive1} into \eqref{112301}, we obtain the equations that $\{\psi_i\}_{i=1}^n$ should satisfy
\begin{align}\label{pppp01}
\begin{cases}
\psi_1'(t)-(cm e^{r(T-t)}+ \frac{(2\beta+1)(\mu - r)^2}{4 r}[1-e^{2\beta r(t-T)}]+\lambda_1)\psi_1(t)+\lambda_1 \psi_2(t)=0,\\
\psi_2'(t)-(cm e^{r(T-t)}+ \frac{(2\beta+1)(\mu - r)^2}{4 r}[1-e^{2\beta r(t-T)}]+\lambda_2)\psi_2(t)+\lambda_2 \psi_3(t)=0,\\
\cdots\\
\psi_{n-1}'(t)-(cm e^{r(T-t)}+ \frac{(2\beta+1)(\mu - r)^2}{4 r}[1-e^{2\beta r(t-T)}]+\lambda_{n-1})\psi_{n-1}(t)+\lambda_{n-1} \psi_n(t)=0,\\
\psi_n'(t)-(cm e^{r(T-t)}+ \frac{(2\beta+1)(\mu - r)^2}{4 r}[1-e^{2\beta r(t-T)}]+\lambda_n)\psi_n(t)+\lambda_n\psi_1(t)\mathbb{E}( e^{mYe^{r(T-t)}})=0.
\end{cases}
\end{align}
The boundary condition of $\{\psi_i\}_{i=1}^n$ is
\begin{align}\label{turywew1}
\psi_i(T)=1,i=1,2, \cdots, n.
\end{align}
We need to solve for a  continuously differentiable and non-negative solution to \eqref{pppp01}-\eqref{turywew1}. The non-negativity of the solution is to make sure that the solution of \eqref{112301} is concave.  We apply change of variables to simplify \eqref{pppp01}.

Define
\begin{equation}
F(t)=\frac{c m}{r} e^{r T}\left(1-e^{-r t}\right)+\frac{(2 \beta+1)(\mu-r)^{2}}{4 r} t+\frac{(2 \beta+1)(\mu-r)^{2}}{8 \beta r^{2}} e^{-2 \beta r T}\left(1-e^{2 \beta r t}\right).
\end{equation}
Multiplying $e^{-F(t)}$ on both sides of \eqref{pppp01} gives
\begin{align}\label{4.9}
\begin{cases}
e^{-F(t)}\psi_1'(t)-e^{-F(t)}\left(cm e^{r(T-t)}+ \frac{(2\beta+1)(\mu - r)^2}{4 r}[1-e^{2\beta r(t-T)}]+\lambda_1\right)\psi_1(t)\\
\quad\quad\quad\quad\quad+\lambda_1 \psi_2(t)e^{-F(t)}=0,\\
e^{-F(t)}\psi_2'(t)-e^{-F(t)}\left(cm e^{r(T-t)}+ \frac{(2\beta+1)(\mu - r)^2}{4 r}[1-e^{2\beta r(t-T)}]+\lambda_2\right)\psi_2(t)\\
\quad\quad\quad\quad\quad+\lambda_2 \psi_3(t)e^{-F(t)}=0,\\
\ldots\\
e^{-F(t)}\psi_{n-1}'(t)-e^{-F(t)}\left(cm e^{r(T-t)}+ \frac{(2\beta+1)(\mu - r)^2}{4 r}[1-e^{2\beta r(t-T)}]+\lambda_{n-1}\right)\psi_{n-1}(t)\\
\quad\quad\quad\quad\quad\quad+\lambda_{n-1} \psi_n(t)e^{-F(t)}=0,\\
e^{-F(t)}\psi_n'(t)-e^{-F(t)}\left(cm e^{r(T-t)}+ \frac{(2\beta+1)(\mu - r)^2}{4 r}[1-e^{2\beta r(t-T)}]+\lambda_n\right)\psi_n(t)\\
\quad\quad\quad\quad\quad+\lambda_n\psi_1(t)\mathbb{E}( e^{mYe^{r(T-t)}})e^{-F(t)}=0.
\end{cases}
\end{align}
Denote
\begin{align*}\label{434352311}
\varphi_i(t):=e^{-F(t)}\psi_i(t), i=1,2,\ldots, n,
\end{align*}
\eqref{4.9} turns to be
\begin{align}
\begin{cases}
\varphi_1'(t)-\lambda_1\varphi_1(t)+\lambda_1 \varphi_2(t)=0,\\
\varphi_2'(t)-\lambda_2\varphi_2(t)+\lambda_2\varphi_3(t)=0,\\
\cdots\\
\varphi_{n-1}'(t)-\lambda_{n-1}\varphi_{n-1}(t)+\lambda_{n-1}\varphi_n(t)=0,\\
\varphi_n'(t)-\lambda_n \varphi_n(t)+\lambda_n \mathbb{E}(e^{mY e^{r(T-t)}})\varphi_1(t)=0,
\end{cases}
\end{align}
or equivalently,
\begin{align}\label{78e0e03}
\left(
  \begin{array}{c}
    \varphi_1'(t) \\
    \varphi_2'(t) \\
    \vdots\\
    \varphi_{n-1}'(t) \\
    \varphi_n'(t) \\
  \end{array}
\right)=Q(t)\left(
                 \begin{array}{c}
                   \varphi_1(t) \\
                   \varphi_2(t) \\
                   \vdots \\
                 \varphi_{n-1}(t) \\
                 \varphi_n(t) \\
                 \end{array}
               \right),
\end{align}
where
\begin{align*}
Q(t)=\left(
       \begin{array}{cccccc}
         \lambda_1 & -\lambda_1 & 0 & \cdots & 0 & 0 \\
          0&  \lambda_2& -\lambda_2 &\cdots  &0 & 0 \\
          0&  0&\lambda_3  & \cdots & 0 & 0 \\
          \vdots& \vdots & \vdots & \vdots & \vdots & \vdots \\
          0& 0 & 0 & \cdots & \lambda_{n-1} &-\lambda_{n-1}  \\
        -\lambda_n z(t) &0  &0  &\cdots  & 0 & \lambda_n \\
       \end{array}
     \right)
\end{align*}
is a $n\times n$ matrix and $z(t):=\mathbb{E}(e^{mY e^{r(T-t)}})$. We assume that $z(t)<+\infty$ for all $t\in[0,T]$.
The boundary condition is
\begin{align} \label{7878232}
\varphi_i(T)=e^{-F(T)}>0.
\end{align}

Note that the Cayley--Hamilton theorem is not applicable to prove the non-negativity of each $\varphi_i$ due to the non-exchangeability of $Q$ for different time:
for any $u,t\in[0,T], u \neq t$,
\begin{equation}
Q(t)Q(u)\neq Q(u)Q(t).
\end{equation}
We resort to the fixed point argument to show the existence of the non-negative solutions to \eqref{78e0e03}.

Denote $\bar{z}:=\sup_{s\in[0,T]}|z(s)|$ and $\bar{\lambda}:=\max\{\lambda_1,\cdots,\lambda_n\}$.
For a given $\delta \leq \frac{1}{2\bar{z} \bar{\lambda}}$, we divide the time interval $[0,T]$ into the $N$ subintervals $[0,T-(N-1)\delta],\cdots,[T-k\delta,T-(k-1)\delta]$, such that $T\leq N\delta$. Now consider the system \eqref{78e0e03} in the time interval $[T-k\delta,T-(k-1)\delta]$:
\begin{align}\label{78e0e03.2}
	\left\{\begin{array}{ll}
		\varphi_1'(t)-\lambda_1\varphi_1(t)+\lambda_1 \varphi_2(t)=0,&\\
		\varphi_2'(t)-\lambda_2\varphi_2(t)+\lambda_2\varphi_3(t)=0,&\\
		\cdots  &\quad t \in [T-k\delta,T-(k-1)\delta],\\
		\varphi_{n-1}'(t)-\lambda_{n-1}\varphi_{n-1}(t)+\lambda_{n-1}\varphi_n(t)=0,&\\
		\varphi_n'(t)-\lambda_n \varphi_n(t)+\lambda_n z(t) \varphi_1(t)=0.&
	\end{array}\right.
\end{align}
Consider the non-negative valued continuous function space $C([T-k\delta,T-(k-1)\delta];\mathbb{R}_{+})$ equipped with the supremum norm $\|v\|_{\infty} := \sup_{s\in[T-k\delta,T-(k-1)\delta]}|v(s)|$ which is a Banach space.
If $k=1$, then the boundary condition is $\varphi_i(T-(k-1)\delta)=e^{-F(T)}>0$. If $k \neq 1$, we assume that the same system defined in the former time interval $[T-(k-1)\delta,T-(k-2)\delta]$ admits a unique non-negative solution $\bar{\varphi}_i(t) \in C([T-(k-1)\delta,T-(k-2)\delta];\mathbb{R}_{\geq0})$, then let $\varphi_i(T-(k-1)\delta):=\bar{\varphi}_i(T-(k-1)\delta)$.

\begin{lem}
	 For any $k=1,\cdots,N$, the equation system \eqref{78e0e03.2} admits a unique solution $(\varphi_1,\cdots,\varphi_n)$ such that $\varphi_i(t) \in C^1([T-k\delta,T-(k-1)\delta];\mathbb{R}_{+})$ for all $i = 1,2,\cdots,n$.
\end{lem}
\begin{proof}
	We decouple the system \eqref{78e0e03.2} by constructing a map $\Phi: C([T-k\delta,T-(k-1)\delta];\mathbb{R}_{+}) \ni \hat{\varphi}_1 \longmapsto \varphi_1 \in C([T-k\delta,T-(k-1)\delta];\mathbb{R})$ as follows
	\begin{align}\label{78e0e03.3}
		\begin{cases}
			\varphi_1'(t)-\lambda_1\varphi_1(t)+\lambda_1 \varphi_2(t)=0,\\
			\varphi_2'(t)-\lambda_2\varphi_2(t)+\lambda_2 \varphi_3(t)=0,\\
			\cdots\\
			\varphi_{n-1}'(t)-\lambda_{n-1}\varphi_{n-1}(t)+\lambda_{n-1}\varphi_n(t)=0,\\
			\varphi_n'(t)-\lambda_n \varphi_n(t)+\lambda_n z(t)\hat{\varphi}_1(t)=0,
		\end{cases}
	\end{align}
	Since $\hat{\varphi}_1(t)$ and $z(t)$ are non-negative functions, it is easy to show that
	\begin{equation}
		\label{solution.Cn}
		\varphi_n(t) = e^{-\lambda_n(T-(k-1)\delta-t)}\bar{\varphi}_n(T-(k-1)\delta) + \lambda_n \int_t^{T-(k-1)\delta} e^{-\lambda_n(s-t)} z(s)\hat{\varphi}_1(s) ds > 0.
	\end{equation}
	Hence $\varphi_n(t)$ belongs to $C([T-k\delta,T-(k-1)\delta];\mathbb{R}_{+})$. We conduct this procedure to the equation of $\varphi_{n-1},\cdots,\varphi_1$ one-by-one. It can be shown that the system of \eqref{78e0e03.3} admits a unique solution $(\varphi_1,\cdots,\varphi_n)$ and all the entries belong to $C([T-k\delta,T-(k-1)\delta];\mathbb{R}_{+})$. Therefore $\Phi$ is a self-map.
	
	If we consider another input $\hat{\varphi}^\varepsilon_1(t)$ to the system,
	\begin{align}\label{78e0e03.4}
		\begin{cases}
			(\varphi^\varepsilon_1)'(t)-\lambda_1\varphi^\varepsilon_1(t)+\lambda_1 \varphi^\varepsilon_2(t)=0,\\
			(\varphi^\varepsilon_2)'(t)-\lambda_2\varphi^\varepsilon_2(t)+\lambda_2\varphi^\varepsilon_3(t)=0,\\
			\cdots\\
			(\varphi^\varepsilon_{n-1})'(t)-\lambda_{n-1}\varphi^\varepsilon_{n-1}(t)+\lambda_{n-1}\varphi^\varepsilon_n(t)=0,\\
			(\varphi^\varepsilon_n)'(t)-\lambda_n \varphi^\varepsilon_n(t)+\lambda_n z(t)\hat{\varphi}^\varepsilon_1(t)=0.
		\end{cases}
	\end{align}
	The last equation can be written as
	\begin{equation}
		\label{solution.Cnepsilon}
		\varphi^\varepsilon_n(t) = e^{-\lambda_n(T-t)}\bar{\varphi}_n(T-(k-1)\delta) + \lambda_n \int_t^{T-(k-1)\delta} e^{-\lambda_n(s-t)} z(s)\hat{\varphi}^\varepsilon_1(s) ds.
	\end{equation}
	We take the difference of \eqref{solution.Cn} and \eqref{solution.Cnepsilon},
	\begin{align*}
		\varphi^\varepsilon_n(t) - \varphi_n(t) = \lambda_n \int_t^{T-(k-1)\delta} e^{-\lambda_n(s-t)} z(s)\left(\hat{\varphi}^\varepsilon_1(s)-\hat{\varphi}_1(s)\right) ds
		\leq \delta \bar{\lambda}\bar{z}\|\hat{\varphi}^\varepsilon_1-\hat{\varphi}_1\|_{\infty}.
	\end{align*}
	For the second-to-last equation, we have
	\begin{align*}
		\varphi^\varepsilon_{n-1}(t) - \varphi_{n-1}(t) =& \lambda_{n-1} \int_t^{T-(k-1)\delta} e^{-\lambda_{n-1}(s-t)} z(s)\left(\varphi^\varepsilon_n(t) - \varphi_n(t)\right) ds \\
		\leq& \delta \bar{\lambda}\bar{z} \|\varphi^\varepsilon_n-\varphi_n\|_{\infty}
		\leq \delta^2 \bar{\lambda}^2\bar{z}^2 \|\hat{\varphi}^\varepsilon_1-\hat{\varphi}_1\|_{\infty}.
	\end{align*}
	We take this procedure to other equations in systems \eqref{78e0e03.3} and \eqref{78e0e03.4} to deduce that
	\begin{equation}
		\|\varphi^\varepsilon_1 - \varphi_1 \|_{\infty}\leq \delta^n \bar{\lambda}^n\bar{z}^n \|\hat{\varphi}^\varepsilon_1-\hat{\varphi}_1\|_{\infty} \leq \frac{1}{2^n}\|\hat{\varphi}^\varepsilon_1-\hat{\varphi}_1\|_{\infty}.
	\end{equation}

Hence the map $\Phi$ is a contractive map. By the Banach fixed point theorem, there exists a unique fixed point $\varphi_1 \in C([T-k\delta,T-(k-1)\delta];\mathbb{R}_{+})$ such that $\Phi(\varphi_1) = \varphi_1$. Until now, we proved that the equation system \eqref{78e0e03.2} admits a unique solution $(\varphi_1,\cdots,\varphi_n)$ such that, for any $i = 1,\cdots,n$, $\varphi_i(t) \in C([T-k\delta,T-(k-1)\delta];\mathbb{R}_{+})$. By the continuity of $z(t)$, the system \eqref{78e0e03.2} further shows that each $\varphi_i(t) \in C^1([T-k\delta,T-(k-1)\delta];\mathbb{R}_{+})$.
\end{proof}

\begin{thm}
	The equation system \eqref{78e0e03} admits a unique solution $(\varphi_1,\cdots,\varphi_n)$ such that $\varphi_i(t) \in C^1([0,T];\mathbb{R}_{+})$ for all $i = 1,2,\cdots,n$.
\end{thm}

\begin{proof}
	We paste all the solutions in the $N$ subintervals $[T-k\delta,T-(k-1)\delta]$ together to obtain the solution to system \eqref{78e0e03}.
\end{proof}

At this point, we have constructed a continuously differentiable concave solution $v(t,x,s,i)$ to the HJB equation \eqref{hjb01001}-\eqref{hjb02001}.
 When we provide a verification theorem  to show that under proper conditions, the continuously differentiable solution to the HJB equation is indeed the optimal value function defined in \eqref{sszx1223}.
\begin{thm}(verification theorem)\label{2546090}
For the case of interest rate not being 0, denote $\iota:=\frac{4(\mu-r)^2}{\sigma^2}$, if either one of the following conditions is satisfied:
\begin{enumerate}
    \item $ \iota\leq \frac{\mu^2}{2\sigma^2}$,
    \item $ \iota>\frac{\mu^2}{2\sigma^2}$,  and  $T<\gamma_3^{-1}\operatorname{arccot}(\frac{-\gamma_1}{\gamma_3})$, where $\gamma_1=\mu \beta $, $\gamma_3=\frac{\sqrt{-4\mu^2\beta^2+8\beta^2 \sigma^2\iota}}{2}$,
\end{enumerate}
then the optimal value function $V (t,x,s,i) = v(t,x,s,i)$, where v is shown in \eqref{112301}. The optimal investment policy is \begin{align}\label{asasss111}\begin{split}
a^*_t
=&\frac{(\mu-r)+(1-e^{2\beta r(T-t)})\frac{(\mu-r)^2}{2 r} }{\sigma^2 s^{2\beta}me^{r(T-t)}}.
\end{split}
\end{align}
\end{thm}
Detailed proof is shown in the Appendix \ref{appendixC}.
\section{Sensitivity Analysis}\label{5}
This section explores  sensitivity of the optimal policy and the optimal value function on model parameters. We assume that the interarrival times follow Erlang (2) distribution. Unless otherwise stated, the parameters are shown in the following table.

\begin{table}[htbp]
	\centering

	\begin{tabular}{cccccccccc}
		\toprule  
		$\mu$&$r$&$\beta$&$T$&$\lambda_1$ &$\lambda_2$ & $\sigma$&$m$&$c$&$s$\\
		\midrule  
		0.2&0.18&1&2&0.5&2& 0.3&1&2.5&1\\
		\bottomrule  
	\end{tabular}
\end{table}

\begin{figure}[htbp]
\begin{minipage}[t]{0.5\linewidth}
\centering
\includegraphics[height=5cm]{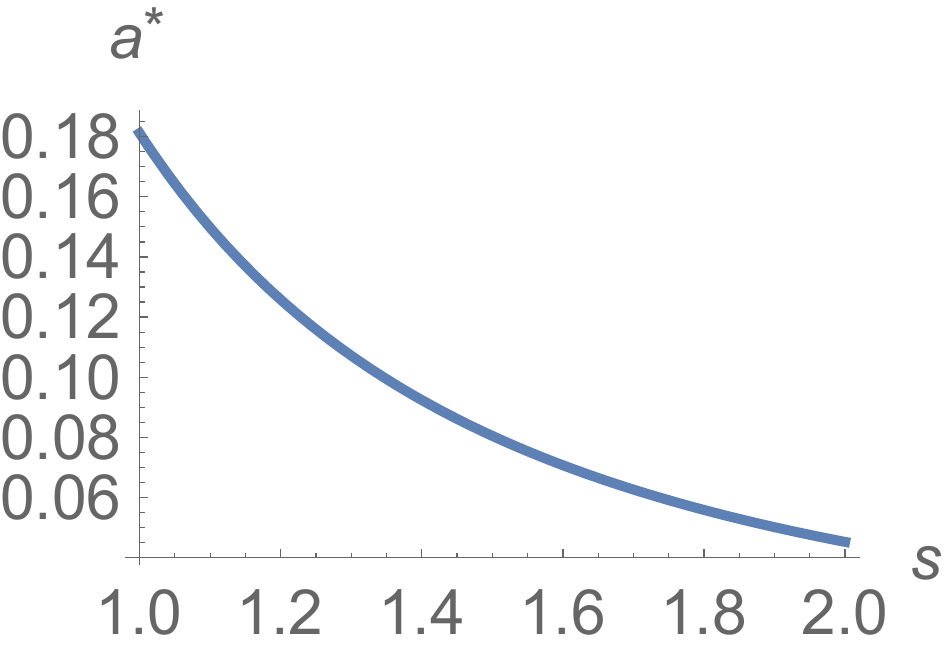}
\caption{The optimal strategy $a^*$ about \\ stock price $s$ at time $t=1.$}
  \label{fig:side:a1}
\end{minipage}%
\begin{minipage}[t]{0.5\linewidth}
\centering
\includegraphics[height=5cm]{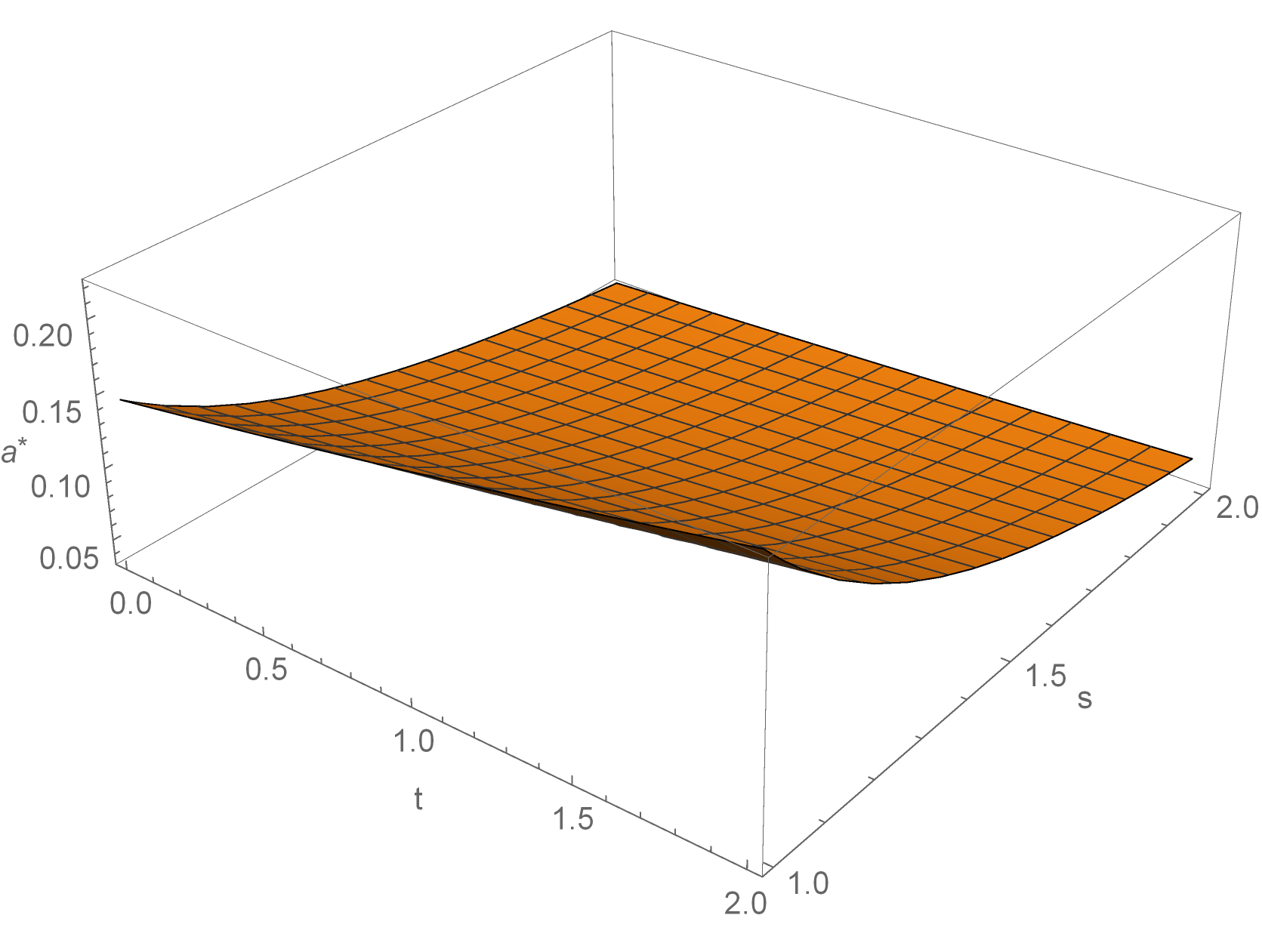}
\caption{The optimal strategy $a^*$ about \\ stock price $s$ and time $t$.}
  \label{fig:side:a2}
\end{minipage}
\end{figure}

From the explicit expression we can see  that the optimal policy is independent of the current surplus $x$, the current phase $i$. The optimal policy depends on  current time $t$, the stock price $s$, the interest rate $r$, the drift $\mu $, and the volatility $\sigma $ of the stock price.
Figure \ref{fig:side:a1} shows the effect of stock process $s$ on the  optimal investment policy $a^*$ when $t=1$. The investment amount on stocks  is decreasing in the  stock price. This phenomenon is in line with intuition since as stock price increasing, the cost of holding stocks will be higher. The optimal investment strategy is a ``buy low, sell high" strategy.
Figure \ref{fig:side:a2} is a 3-dimensional picture showing the effect of time $t$ and stock price $s$ on the optimal investment policy $a^*$, from which we can see that the amount of investment  increases as the time $t$ increases. This  means as $t$ approaches to the terminal time, the insurer tends to take more risk in order to get higher returns.

Next, we  analyze the effects of $s$ and $t$ on the optimal strategy in the case that  the claims follow a uniform and exponential distribution, respectively.
\begin{example}
 When claim size follows the uniform distribution  on $[0,1]$.
 \end{example}
\begin{figure}[ht]
    \centering
    \begin{minipage}[b]{0.45\textwidth}
        \centering
        \includegraphics[width=\textwidth]{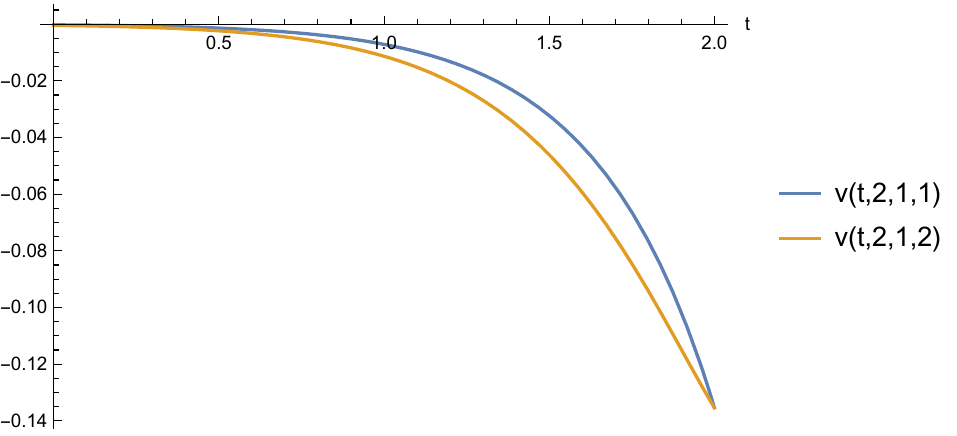}
          \caption{The value function $V$ about time $t$ at $x=2$ and  $s=1$.}
        \label{5748579432a1}
    \end{minipage}
    \quad
    \begin{minipage}[b]{0.45\textwidth}
        \centering
        \includegraphics[width=\textwidth]{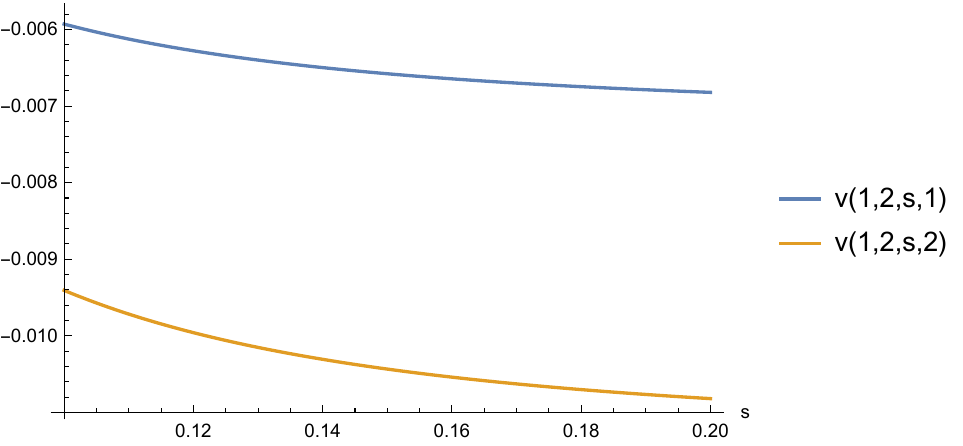}
          \caption{The value function $V$ about stock price $s$  at $t=1$ and  $x=2$ .}
        \label{uvs}
    \end{minipage}
\end{figure}
Figure \ref{5748579432a1}  shows the  the effect of $t$  on the optimal value function at $x=2, s=1$.  It is obvious that  the value function is decreasing with respect to $t$. Besides, we can also see that  $V(t,x,s,1)\ge V(t,x,s,2)$.
This is because  that the  insurer in phase  $2$ is more ``close" to the claim time than that in phase 1,  if the insurer is in phase $2$, it is expected to face the claims in the upcoming future.

Figure \ref{uvs}  shows the value of $V$ against the stock price  $s$ at $t=1$ and $x=2$. $V$ is decreasing with respect to $s$. This is because the volatility of CEV stock price increases as the stock price increases, leading to the decrease of value function.
\begin{example}
When claim size follows the exponentially distributed with the probability density function  $f(Y)$ satisfying  $f(Y)=\kappa e^{-\kappa Y}$ , where $\kappa=2$ denotes the parameter of the exponential distribution.
 \end{example}
\begin{figure}[ht]
    \centering
    \begin{minipage}[b]{0.45\textwidth}
         \centering
        \includegraphics[width=\textwidth]{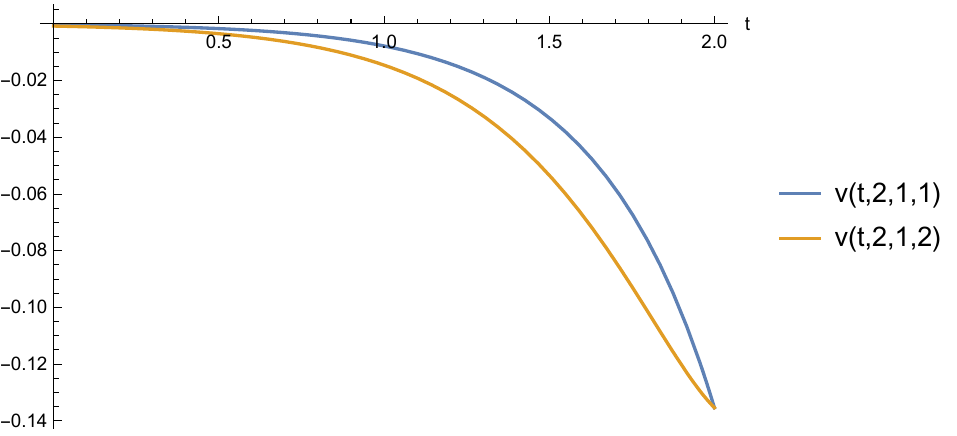}
        \caption{The value function $V$ about time $t$ at $x=2$ and  $s=1$.}
    \label{5748579432a2}
    \end{minipage}
    \quad
    \begin{minipage}[b]{0.45\textwidth}
        \centering
        \includegraphics[width=\textwidth]{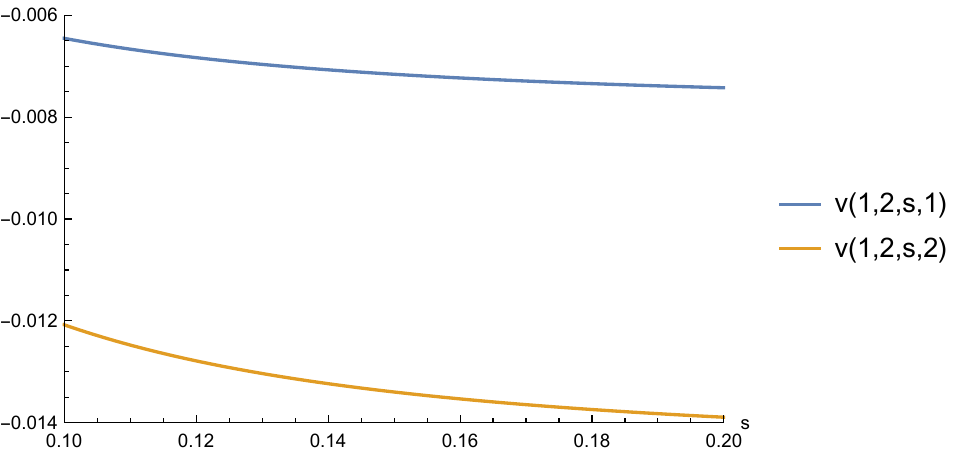}
        \caption{The value function $V$ about stock price $s$ at $t=1$ and  $x=2$.}
    \label{evs}
    \end{minipage}
\end{figure}

Figure  \ref{5748579432a2}  shows the  the effect of $t$   on the optimal value function  at $x=2, s=1$. We can see that the trend of $V$ with respect to $t$ is the same as that shown in Figures \ref{5748579432a1} and we also have $V(t,x,s,1)\ge V(t,x,s,2)$. Figure \ref{evs} shows the value of $V$ against the stock price $s$ at $t=1$ and $x=2$. The value function on Figure \ref{evs} is less than that on Figure \ref{uvs}. Although the exponential distribution with $\kappa=2$ has the same mean value $\frac{1}{2}$ as the uniform distribution on [0,1], it is possible for exponential distributed variable to take a very large value, namely a much larger value than U[0,1]. This exponential distributed claim size may increase the risk of insurance company, and thus reduce its utility.

\section{Conclusion}\label{conclusion}
This article considers the optimal investment of the renewal surplus process in which the interclaim times are Erlang($n$) distributed. By applying the  Laplace transform to the Markov chain behind the Erlang distribution, and applying  the decoupling argument and Banach fixed point theorem, we provide the  concavity of the value function which establishes the well-possedness of the optimization problem. Our results are explicit and semi-explicit optimal strategies and value functions. The optimal strategies are shown to be independent of wealth process and the current phase of the state which is consistent with the optimal  strategy  of \cite{Hipp2000} in their case.  The utility of the insurer  decreases when  time  is  close to the  next claim or the stock price is increasing.

\begin{appendices}
\section{Proof of Theorem \ref{e2daw21212}}\label{appendixA}
Denote  $\zeta:=\mathbb{E}(e^{mY})$, we consider the  matrix $\mathbb{O}$ in  the form of
\[\left(
       \begin{array}{cccccc}
         -\lambda_1 & \lambda_1 & 0 & \cdots & 0 & 0 \\
          0&  -\lambda_2& \lambda_2 &\cdots  &0 & 0 \\
          0&  0&-\lambda_3  & \cdots & 0 & 0 \\
          \vdots& \vdots & \vdots & \vdots & \vdots & \vdots \\
          0& 0 & 0 & \cdots & -\lambda_{n-1} &\lambda_{n-1}  \\
        \lambda_n \zeta &0  &0  &\cdots  & 0 & -\lambda_n \\
       \end{array}
     \right),\]
and show that, for any $t>0$,  every element of the matrix $e^{\mathbb{O}t}$ is non-negative.
Consider a  Markov chain $\tilde{J}_t$ with the  transition matrix \[\left(
       \begin{array}{cccccc}
         -\lambda_1 & \lambda_1 & 0 & \cdots & 0 & 0 \\
          0&  -\lambda_2& \lambda_2 &\cdots  &0 & 0 \\
          0&  0&-\lambda_3  & \cdots & 0 & 0 \\
          \vdots& \vdots & \vdots & \vdots & \vdots & \vdots \\
          0& 0 & 0 & \cdots & -\lambda_{n-1} &\lambda_{n-1}  \\
        \lambda_n\zeta  &0  &0  &\cdots  & 0 & -\lambda_n\zeta \\
       \end{array}
     \right).\]
From the transition matrix, we can see that
\[\mathbb{P}(\tilde{J}_t=n, t\in[0,s]|\tilde{J}_0=n)=e^{-\lambda_n \zeta s}=1-\lambda_n\zeta s+o(s),
\]
 \[\mathbb{P}(\tilde{J}_{[0,s]}=i |\tilde{J}_0=i)=1-\lambda_i s+o(s), i\neq n. \]
If the initial phase $i\neq n$, define $f_{ij}(t):=\mathbb{E}\left[\mathbf{1}_{\{X_t=j\}}|X_0=i\right]$, then
\begin{align*}
f_{ij}(t)=(1-\lambda_is)f_{ij}(t-s)+\lambda_i s f_{i+1, j}(t-s)+o(s).
\end{align*}
Thus, \[\frac{f_{ij}(t)-f_{ij}(t-s)}{s}=\frac{-\lambda_isf_{ij}(t-s)+\lambda_i s f_{i+1, j}(t-s)+o(s)}{s},\]
Letting $s\downarrow  0$ gives \begin{align}\label{34356412}
f_{ij}'(t)=-\lambda_i f_{ij}(t)+\lambda_{i}f_{i+1, j}.
\end{align}
If the initial phase $i=n$, define $f_{ij}(t)=\mathbb{E}\left[e^{\lambda_i (\zeta-1)\int_0^t\mathbf{1}_{\{\tilde{J}_u=i\}}du}\mathbf{1}_{\{X_t=j\}}|X_0=i\right]$.  We can obtain that
\begin{align}\label{787312ds}
f_{nj}(t)=e^{\lambda_n (\zeta-1)s}(1-\lambda_n \zeta s)f_{nj}(t-s)+f_{1j}(t)\lambda_n \zeta s+o(s).
\end{align}
Note that
\begin{align}\label{77974821}
e^{\lambda_n (\zeta-1)s}=1+\lambda_n (\zeta-1)s+o(s).
\end{align}
Substituting \eqref{77974821} into \eqref{787312ds} gives
\begin{align*}
\begin{split}
f_{nj}(t)&=(1+\lambda_n (\zeta-1)s)(1-\lambda_n \zeta s)f_{nj}(t-s)+f_{1j}(t)\lambda_n \zeta s+o(s)\\
&=(1-\lambda_n  s)f_{nj}(t-s)+f_{1j}(t)\lambda_n \zeta s+o(s),
\end{split}
\end{align*}
which further gives
\begin{align}\label{5647232}
\frac{f_{nj}(t)-f_{nj}(t-s)}{s}=\frac{-\lambda_n  sf_{nj}(t-s)+f_{1j}(t)\lambda_n \zeta s+o(s)}{s}.
\end{align}
Letting $s\downarrow 0$ in \eqref{5647232} gives
\begin{align}\label{121212}
f_{nj}'(t)=-\lambda_n f_{nj}(t)+f_{1j}\lambda_n \zeta.
\end{align}
Combining \eqref{34356412} and \eqref{121212}, we can see that the matrix  $\{f_{ij}\}_{n\times n }$ is the solution of $f'(t)=\mathbb{O}f(t)$ with boundary condition $f(0)=\mathbf{E}$, where $\mathbf{E}$ is the identity matrix, or in other words, $f(t)=e^{\mathbb{O} t}$. By the definition of $f_{ij}(t)$, we can see that $f_{ij}(t)$ is non-negative, which means that every element of $e^{\mathbb{O} t}$ is non-negative. The proof is complete.

\section{Proof of Theorem \ref{475sss}}\label{appendixB}
Denote $J_t$ the current phase of Erlang ($n$) distribution  at time $t$.
For any strategy $a$,
by It\^{o}'s formula,
\begin{align}\label{impr1}\begin{split}
&v(\tau_n \wedge T,X_{\tau_n \wedge T},S_{\tau_n \wedge T} , J_{\tau_n \wedge T})\\=&v(t,x,s,j)+\int_t^{\tau_n \wedge T}(v_u+(c+a_u\mu )v_{x}+\mu S_uv_s+\frac{1}{2}\sigma^2 S^{2\beta+1}_u v_{ss} \\
&+\sigma^2S^{2\beta+1}_u a_u v_{xs}+\frac{1}{2}v_{xx}\sigma^2 S_u^{2\beta}a_u^2)du+\int_t^{\tau_n \wedge T}\sigma S_u^\beta a_udW_u\\&+\sum_{t\le u\le (\tau_n \wedge T)}(v(u, X_{u}, S_u, J_u)-v(u, X_{u-}, S_u, J_{u-})),
\end{split}
\end{align}
where \begin{align}
\tau_n=n \wedge \inf\{u>t; |X^a_u|\ge n\}\wedge \inf\{u>t; |S_u|\ge n\},
\end{align}
for $n=1,2,\cdots. $ Since $v$ is a continuously differentiable solution to the HJB equation, taking expectation on both sides of \eqref{impr1} leads to
\[\mathbb{E}v(\tau_n \wedge T,X_{\tau_n \wedge T},S_{\tau_n \wedge T},J_{\tau_n \wedge T} )\le v(t,x,s,j).\]
By Fatou's lemma, letting $n\rightarrow +\infty$ gives $$\mathbb{E}[U(X_T^a)]\le v(t,x,s,j ). $$

On the other hand,  we choose the optiaml strategy $a^*$ given by \eqref{r=0a*} and  denote $X^*_t$ the corresponding surplus under strategy $a^*$. To show $\mathbb{E}[U(X_T^*)]= v(t,x,s, j), $  it suffices to prove that
\begin{align}\label{3526352a}\lim_{n\rightarrow +\infty}\mathbb{E}[v(\tau_n \wedge T,X_{\tau_n \wedge T}^*, S_{\tau_n \wedge T}, J_{\tau_n \wedge T})]= v(t,x,s,j).\end{align}
By crystall ball condition, it suffices to prove that $\mathbb{E}[v^2(t,X^*_t,S_t,J_t)]<+\infty$.
Combining \eqref{1123019456de67w2}, we get that
\begin{align*}\begin{split}
&v^2(t,X^*_t,S_t,J_t)\\=&\frac{1}{m^2}\exp\left\{-2m X^*_t +\frac{2\mu^2}{\sigma^2}(t-T)S_t^{-2\beta}\right\}\psi_{J_t}^2(t)\\
\le &M_1 \exp\left\{-2m X^*_t\right\},
\end{split}
\end{align*}
where $M_1>0$ is a constant.  Substituting $a^*$ into \eqref{sa1wqw} gives
\begin{align}\label{sas2323123}
X_t^*=x_0+\int_0^t (\mu a^*_u+c)du+\int_0^t \sigma S_u^\beta a_u^*dW_u-\sum_{0\le u\le t} (X_{u-}-X_u)\chi_{\{\triangle X_u\neq 0\}},
\end{align}
where $\chi$ is the indicator function.
Substituting \eqref{sas2323123} gives
\begin{align}\label{wq7783}\begin{split}
&\exp\left\{-2m X^{*}_t\right\}\\
=&\exp\bigg\{-2m  \bigg( x_0+\int_0^t (\mu a^*_u+c)du\\
 &+\int_0^t \sigma S_u^\beta a^*_u dW_u-\sum_{0\le u\le t} (X_{u-}-X_u)\chi_{\{\triangle X_u\neq 0\}}\bigg)\bigg\}
 \\
 \le &\exp\bigg\{-2m  \bigg( x_0+\int_0^t (\mu a^*_u+c)du+\int_0^t \sigma S_u^\beta a^*_u dW_u\bigg)\bigg\}
 \\
 \le& M_2\exp\bigg\{-2m  \bigg( \int_0^t \mu a^*_udu+\int_0^t \sigma S_u^\beta a^*_u dW_u\bigg)\bigg\}, \end{split}
\end{align}
where $M_2>0$ is a constant.
Substituting $a^*_t=\frac{\mu +\mu^2\beta(T-t)}{\sigma^{2} {S_t}^{2 \beta} m}$ into the above inequality, we obtain that
\begin{align}\begin{split}
&\exp\left\{-2m X^*_t\right\}\\
\le &M_2\exp\bigg\{-2   \int_0^t \frac{\mu^2+\mu^3\beta(T-u)}{\sigma^2}S^{-2\beta}_u du-2\int_0^t\frac{\mu+\mu^2\beta(T-u)}{\sigma }S_u^{-\beta}dW_u\bigg\}.
\end{split}
\end{align}
Note that $H_u=\frac{\mu^2+\mu^3\beta(T-u)}{\sigma^2},$ 
Then
\begin{align}
\begin{split}
\exp\left\{-2m X^*_t\right\}\le& M_2\exp\bigg\{-2\int_0^tH_uS_u^{-2\beta}du-2\int_0^t\frac{\sigma}{\mu}H_uS_u^{-\beta}dW_u\bigg\}
\\
=&M_2\exp\bigg\{-2\int_0^tH_uS_u^{-2\beta}du+4\int_0^t\frac{\sigma^2}{\mu^2}H^2_uS_u^{-2\beta}du\bigg\}\\
&\cdot\exp\bigg\{-2\int_0^t\frac{\sigma}{\mu}H_uS_u^{-\beta}dW_u-4\int_0^t\frac{\sigma^2}{\mu^2}H^2_uS_u^{-2\beta}du\bigg\}
\\:=&M_2\exp\{{H}_{1t}+{H}_{2t}\},
\end{split}
\end{align}
where ${H}_{1t}=-2\int_0^tH_uS_u^{-2\beta}du+4\int_0^t\frac{\sigma^2}{\mu^2}H^2_uS_u^{-2\beta}du$, ${H}_{2t}=-2\int_0^t\frac{\sigma}{\mu}H_uS_u^{-\beta}dW_u-4\int_0^t\frac{\sigma^2}{\mu^2}H^2_uS_u^{-2\beta}du$.
By H\"{o}lder's inequality, we obtain that
\begin{align}
\mathbb{E}\left(\exp\left\{-2m X^*_t\right\}\right)\le M_2\left(\mathbb{E}\left(e^{2H_{1t}}\right)
\right)^{\frac{1}{2}}\left(\mathbb{E}\left(e^{2{H}_{2t}}\right)
\right)^{\frac{1}{2}}.
\end{align}
Notice that $\Gamma =\sup_{u \in [0,T]}\left\{-4H_u+8\frac{\sigma^2}{\mu^2}H^2_u\right\}=\frac{4 \mu^{2}+12 \mu^{3} \beta T+8 \mu^{4} \beta^{2} T^{2}}{\sigma^{2}},$ applying Theorem 5.1 of \cite{zeng2013} gives
 $$\mathbb{E}\left(e^{2{H}_{1t}}\right)<+\infty.$$
By Lemma 4.3 of \cite{zeng2013},  it is obtained that $e^{2{H}_{2t}}$ is a martingale, therefore $$\mathbb{E}\left(e^{2{H}_{2t}}\right)<+\infty.$$

As this point, we have shown that under the condition, the solution to the HJB equation is indeed the value function of the optimization problem. The optimal policy is given in terms of the value function as shown in \eqref{23u2i2323}.

\section{Proof of Theorem \ref{2546090}}\label{appendixC}

Similar to the proof of Theorem \ref{475sss}, it is sufficient to show that
if the optimal strategy $a^*$ defined  in \eqref{asasss111} is applied,  it holds that
\[\lim_{n\rightarrow +\infty}\mathbb{E}v(\tau_n \wedge T,X_{\tau_n \wedge T}^*, S_{\tau_n \wedge T}, J_{\tau_n \wedge T})= v(t,x,s,j), j=1,2,\cdots, n.\]
Thus, it requires that, $V(\tau_n \wedge T,X_{\tau_n \wedge T}^*,S_{\tau_n \wedge T},J_{\tau_n \wedge T} )$ is uniformly integrable. By calculation, for some constants $\hat{M}_1>0, \hat{M}_2>0$,
\begin{align}\label{sasa12112dd1}\begin{split}
&v^2(t,X^*_t,S_t,J_t)\\=&\frac{1}{m^2}\exp\left\{-2m X^*_t e^{r(T-t)}-\frac{2(\mu - r)^2}{4\sigma^2\beta r}[1-e^{2\beta r(t-T)}]S_t^{-2\beta}\right\}\psi_{J(t)}^2(t)\\\le &
\hat{M}_1\exp\left\{-2m X^*_t e^{r(T-t)}-\frac{(\mu - r)^2}{2\sigma^2\beta r}[1-e^{2\beta r(t-T)}]S_t^{-2\beta}\right\}\\
\le &\hat{M}_2 \exp\left\{-2m  e^{r(T-t)}X^*_t\right\}.
\end{split}
\end{align}
Substituting $a^*$ into \eqref{sa1wqw} gives
\begin{align}\label{sas23231}\begin{split}
X_t^*=&e^{rt}x_0+\int_0^t e^{r(t-u)}[(\mu-r)a^*_u+c]du+\int_0^t e^{r(t-u)}\sigma S_u^\beta a_u^*dW_u\\
&-\sum_{0\le u\le t} e^{r(t-u)}(X_{u-}-X_u)\chi_{\{\triangle X_u\neq 0\}}.
\end{split}
\end{align}
Substituting \eqref{asasss111} and \eqref{sas23231} into \eqref{sasa12112dd1} gives
\begin{align*}\begin{split}
 &\exp\left\{-2m  e^{r(T-t)}X^*_t\right\}\\
 =&\exp\bigg\{-2m  e^{r(T-t)}\bigg( e^{rt}x_0+\int_0^t e^{r(t-u)}[(\mu-r)a^*_u+c]du\\
 &+\int_0^t e^{r(t-u)}\sigma S_u^\beta a^*_u dW_u-\sum_{0\le u\le t} e^{r(t-u)}(X_{u-}-X_u)\chi_{\{\triangle X_u\neq 0\}}\bigg)\bigg\}\\
 \le&\hat{M}_3\exp\bigg\{-2 \int_0^t (\mu-r)\frac{(\mu-r)+(1-e^{2\beta r(T-u)})\frac{(\mu-r)^2}{2 r} }{\sigma^2 }S_u^{-2\beta}du\\
 &-2 \int_0^t \sigma\frac{(\mu-r)+(1-e^{2\beta r(T-u)})\frac{(\mu-r)^2}{2 r} }{\sigma^2 }S_u^{-\beta}dW_u\bigg\}\\
 =&\hat{M}_3\exp\bigg\{\int_0^t-2 (\mu-r)\tilde{H}_{u}S_u^{-2\beta}du+\int_0^t4\sigma^2\tilde{H}_{u}^2S_u^{-2\beta}du\\&+\int_0^t-2 \sigma \tilde{H}_{u}S_u^{-\beta}dW_u-\int_0^t4\sigma^2 \tilde{H}_{u}^2S_u^{-2\beta}du\\:=&M_3\exp\{\tilde{H}_{1t}+\tilde{H}_{2t}\},
 \end{split}
\end{align*}
where $\hat{M}>0$ is a  constant and
\begin{align*}
&\tilde{H}_u=\frac{(\mu-r)+(1-e^{2\beta r(T-u)})\frac{(\mu-r)^2}{2 r} }{\sigma ^2},\\
&\tilde{H}_{1t}=\int_0^t-2 (\mu-r)\tilde{H}_{u}S_u^{-2\beta}du+\int_0^t4\sigma^2\tilde{H}_{u}^2S_u^{-2\beta}du,
\\
&\tilde{H}_{2t}=\int_0^t-2 \sigma \tilde{H}_{u}S_u^{-\beta}dW_u-\int_0^t4\sigma^2\tilde{H}_{u}^2S_u^{-2\beta}du.
\end{align*}
By Cauchy-Schwarz inequality, we have
\begin{align*}
\mathbb{E}\big(\exp\left\{-2m  e^{r(T-t)}X^*_t\right\}\big)\le M_3\mathbb{E}\big(e^{\tilde{H}_{1t}+\tilde{H}_{2t}}\big)\le M_3\big(\mathbb{E}(e^{2\tilde{H}_{1t}})\big)^{\frac{1}{2}}\big(\mathbb{E}(e^{2\tilde{H}_{2t}})\big)^{\frac{1}{2}}.
\end{align*}
Notice that $\iota=\sup_{u\in[0,T]}\left\{ -4 (\mu-r)\tilde{H}_{u} +8\sigma^2\tilde{H}_{u}^2\right\}=\frac{4(\mu-r)^2}{\sigma^2}$, applying Theorem 5.1 of \cite{zeng2013} gives
%
\begin{align*}
\mathbb{E}(e^{2\tilde{H}_{1t}})<+\infty.
\end{align*}
By Lemma 4.3 of \cite{zeng2013}, we see that $e^{2\tilde{H}_{2t}}$ is a martingale, then
\begin{align*}
\mathbb{E}(e^{2\tilde{H}_{2t}})<+\infty.
\end{align*}

As this point, we have shown that under the condition, the solution to the HJB equation is indeed the value function of the optimization problem. The optimal policy is given in terms of the value function as shown in \eqref{657657}.

\end{appendices}
\noindent {\\ \bf Acknowledgmen} This research is supported by the National Natural Science Foundation of China under grant No.12201104.
\noindent {\\ \bf Disclosure statement and Data availability statement} The authors report that there are no competing interests to declare and there is no data set associated with the paper.

\end{document}